\documentclass{amsart}

\usepackage{cmap}

\usepackage[T1]{fontenc}

\usepackage[utf8x]{inputenc}

\usepackage{xspace}
\usepackage{cite}
\usepackage{amsmath,amssymb,amsfonts}
\usepackage{amsthm}  
\usepackage{graphicx}
\usepackage{textcomp}
\usepackage{xcolor}
\usepackage{flushend}
\usepackage[pdfstartview=FitH, bookmarks, pdfpagemode=UseOutlines,%
  hidelinks, plainpages=false, pdfpagelabels, pdfpagelayout=OneColumn,%
  naturalnames=true, hypertexnames=false]{hyperref}
\usepackage{nicefrac}
\usepackage{mathtools}
\usepackage{soul}
\usepackage{etoolbox}
\usepackage[inline]{enumitem}
\usepackage{twoopt}
\usepackage{amsaddr}

\setlist%
{%
 topsep=0pt,%
 labelsep=6pt,
 noitemsep,%
 leftmargin=*
}

\newtheorem{theorem}{Theorem}[section]
\newtheorem{corollary}{Corollary}[theorem]
\newtheorem{lemma}[theorem]{Lemma}

\newtheorem{definition}[theorem]{Definition}
\newtheorem{example}[theorem]{Example}

\theoremstyle{remark}
\newtheorem{remark}[theorem]{Remark}

\numberwithin{equation}{section}

\providecommand*{\Xmath}[1]{\ensuremath{#1}\xspace}

\providecommand*{\XmathThreePar}[3]{%
  \ifthenelse{\equal{#3}{'}}{%
    \Xmath{#1_{#2}'}
  }{%
    \Xmath{#1_{#2}^{#3}} 
  }
}

\providecommand*{\XmathFourPar}[4]{\XmathThreePar{#1{#2}}{#3}{#4}}
\newcommand*{\Xop}[1]{\ensuremath{\mathrm{#1}\,}}

\newcommand*{\NSet}[3]{\XmathFourPar{\mathbb}{#1}{#2}{#3}}
\newcommandtwoopt*{\R}[2][][]{\NSet{R}{#1}{#2}}
\newcommandtwoopt*{\C}[2][][]{\NSet{C}{#1}{#2}}
\newcommandtwoopt*{\Hset}[2][][]{\NSet{H}{#1}{#2}}

\renewcommand*{\Re}{\Xop{Re}}
\renewcommand*{\Im}{\Xop{Im}}
\newcommand*{\Abs}[2][]{\XBrace[#1]{\lvert}{\rvert}{#2}}
\newcommand*{\XBrace}[4][]{\Xmath{\mathopen#1#2#4\mathclose#1#3}}
\newcommand*{\RBrace}[2][]{\XBrace[#1]{(}{)}{#2}}

\newcommand{\ran}{\operatorname{ran}}
\newcommand{\ip}[2]{\langle #1,#2\rangle}

\newcommand*{\Quaternion}[3]{\XmathFourPar{}{#1}{#2}{#3}}
\newcommand*{\RealPart}[1]{\Re{#1}}
\newcommand*{\ImaginaryPart}[1]{\Im{#1}}

\newcommandtwoopt*{\qa}[2][][]{\Quaternion{a}{#1}{#2}}
\newcommandtwoopt*{\qb}[2][][]{\Quaternion{b}{#1}{#2}}
\newcommandtwoopt*{\qc}[2][][]{\Quaternion{c}{#1}{#2}}
\newcommandtwoopt*{\qi}[2][][]{\Quaternion{i}{#1}{#2}}
\newcommandtwoopt*{\qj}[2][][]{\Quaternion{j}{#1}{#2}}
\newcommandtwoopt*{\qk}[2][][]{\Quaternion{k}{#1}{#2}}
\newcommandtwoopt*{\qp}[2][][]{\Quaternion{p}{#1}{#2}}
\newcommandtwoopt*{\qq}[2][][]{\Quaternion{q}{#1}{#2}}
\newcommandtwoopt*{\qw}[2][][]{\Quaternion{w}{#1}{#2}}
\newcommandtwoopt*{\qx}[2][][]{\Quaternion{x}{#1}{#2}}
\newcommandtwoopt*{\qy}[2][][]{\Quaternion{y}{#1}{#2}}

\newcommandtwoopt*{\rea}[1][][]{\RealPart{a}{#1}}
\newcommandtwoopt*{\reb}[1][][]{\RealPart{b}{#1}}
\newcommandtwoopt*{\rec}[1][][]{\RealPart{c}{#1}}
\newcommandtwoopt*{\req}[1][][]{\RealPart{q}{#1}}
\newcommandtwoopt*{\rex}[1][][]{\RealPart{x}{#1}}

\newcommandtwoopt*{\ima}[2][][]{\Quaternion{\ImaginaryPart{a}}{#1}{#2}}
\newcommandtwoopt*{\imb}[2][][]{\Quaternion{\ImaginaryPart{b}}{#1}{#2}}
\newcommandtwoopt*{\imc}[2][][]{\Quaternion{\ImaginaryPart{c}}{#1}{#2}}
\newcommandtwoopt*{\imp}[2][][]{\Quaternion{\ImaginaryPart{p}}{#1}{#2}}
\newcommandtwoopt*{\imq}[2][][]{\Quaternion{\ImaginaryPart{q}}{#1}{#2}}
\newcommandtwoopt*{\imx}[2][][]{\Quaternion{\ImaginaryPart{x}}{#1}{#2}}
\newcommandtwoopt*{\imy}[2][][]{\Quaternion{\ImaginaryPart{y}}{#1}{#2}}

\begin{document}

\title{Minimum-Norm and Least-Squares Solutions of the Singular Sylvester Equation in Quaternions}

\author{	Hristina Radak, Christian Scheunert, and Frank H. P. Fitzek}

\address{H.~Radak and F.~Fitzek are with the Deutsche Telekom Chair of Communication Networks, TUD Dresden University of Technology, 01062 Dresden, Germany, Email: hristina.radak@tu-dresden.de}

\address{C.~Scheunert is with the Chair of Information Theory and Machine
 Learning, TUD Dresden University of Technology, 01062 Dresden, Germany}

\address{F.~Fitzek is also with the Centre for Tactile Internet with Human-in-the-Loop (CeTI), TUD Dresden University of Technology, 01062 Dresden,  Germany}

\maketitle

\begin{abstract}
The quaternionic Sylvester equation $ax - xb = c$ is singular exactly when
$a$ and $b$ are similar. In that case the homogeneous equation has a
two-dimensional solution space, and the inhomogeneous equation is solvable
only for $c$ in a two-dimensional subspace. Tian (2004) gave the general
solution in this case but left open which solution has minimal norm and what
the least-squares solutions are when the equation is not solvable. We answer
both questions by treating $S(x) = ax - xb$ as a real-linear function on the
four-dimensional Euclidean space of quaternions and working entirely within
the quaternion algebra, so that every object is obtained as an explicit
quaternionic expression. The solution space of the homogeneous equation is
shown to be the complex line through any nonzero solution $x_1$, with the
orthogonal basis $x_1$ and $x_1(\operatorname{Im} b)$ in every case,
including $b = a^{*}$, which the known two-parameter formula excludes. In the
singular case $S$ is skew-symmetric and satisfies
$S^{3} = -4\lvert\operatorname{Im} a\rvert^{2} S$. From these two facts we
obtain the solvability of $ax - xb = c$ as a Fredholm alternative, namely
$c$ orthogonal to the kernel of $S$, which is equivalent to Tian's condition
$ac = cb^{*}$. The Moore--Penrose inverse is
$S^{+} = -S/(4\lvert\operatorname{Im} a\rvert^{2})$, and the orthogonal
projectors onto the range and the kernel of $S$ are given in closed form.
This yields the minimum-norm solution and all least-squares solutions for
every right-hand side $c$, the former being the first term of Tian's general
solution. Finally, the second basis vector $x_1(\operatorname{Im} b)$ is a
real multiple of the quaternion square root of
$(\operatorname{Im} a)(\operatorname{Im} b)^{*}$, which after normalisation
is the unit quaternion rotating $\operatorname{Im} b$ onto
$\operatorname{Im} a$ along the shortest arc.
\end{abstract}


\section{Introduction}\label{sec:INTRO}
 
A general linear function of a quaternion variable $\qx$ has the form
\cite{janovska+opfer2007_on}
\begin{align*}
  S_n(\qx)=\qa[1]\qx\qb[1]+\qa[2]\qx\qb[2]+\dots+\qa[n]\qx\qb[n],
\end{align*}
where the coefficients $\qa[j],\qb[j]$ for $j=1,2,\dots,n$ are nonzero
quaternions. Due to the noncommutativity of the quaternion product,
$S_n$ is in general not homogeneous and hence not linear over the quaternions. However,
$S_n$ does respect real scalar multiplication, so it is real-linear~\cite{janovska+opfer2007_on}. For
$n=2$ and without loss of generality, $S_n$ may be rewritten as
\begin{align}\label{eq:SYLVESTER_FUNCTION}
  S\colon\Hset\to\Hset,\qquad S(\qx):=\qa\qx-\qx\qb ,
\end{align}
which we regard throughout as a real-linear function on the
four-dimensional real vector space $\Hset\cong\R[][4]$ with the Euclidean
inner product \eqref{eq:INNER_PRODUCT}.
The function $S$ is called singular if there exists a nonzero quaternion
$\qx$ for which $S(\qx)=0$. Otherwise, it is called regular. Given a
quaternion $\qc$, the equation
\begin{align}\label{eq:SILVESTER_EQUATION}
  S(\qx)=\qa\qx-\qx\qb=\qc
\end{align}
is referred to as the Sylvester equation. It is called homogeneous if
$\qc=0$ and inhomogeneous otherwise, and it is called singular if the
associated function $S$ is singular and regular otherwise.

Johnson \cite{johnson1944_on} distinguished between the regular and
singular cases and provided topological proofs for the existence and
uniqueness of solutions over division rings, which directly apply to
quaternions. A regular Sylvester equation has a unique solution for every
quaternion $\qc$. Hence, the homogeneous regular Sylvester equation has
only the trivial solution $\qx=0$. Closed-form expressions for the solution
of the inhomogeneous regular Sylvester equation are given in
\cite{janovska+opfer2007_on,porter1997_quaternion,shao+li++2020_basis-free,bolotnikov2015_on}.

The singular Sylvester equation over quaternions poses greater
challenges. For nonreal quaternions
$\qa,\qb$, the function $S$ is singular if and only if $\qa$ and $\qb$ are
similar, that is, $\rea=\reb$ and $\Abs{\qa}=\Abs{\qb}$
\cite{brenner1951_matrices,zhang1997_quaternion,rodman2014_topics}.
In that case $\ker S$ is two-dimensional \cite{janovska+opfer2008_linear}, and
\eqref{eq:SILVESTER_EQUATION} is solvable only for $\qc$ in a two-dimensional
subspace. Rodman \cite{rodman2014_topics} provided a comprehensive
matrix-theoretic foundation for quaternionic linear algebra, including
detailed analyses of eigenvalues and similarity transformations. Tian
provided the general solution of the homogeneous
\cite{tian1999_similarity} and of the inhomogeneous
\cite{tian2004_equations} singular Sylvester equation, the latter by
means of the real $4 \times 4$ matrix representation of quaternions~\cite{gross+trenkler++2001_quaternion}. Shpakivskyi
\cite{shpakivsky2011_linear} reduced the problem to a system of linear
equations over the real number field, while Turner
\cite{turner2006_engineerin} separated real and imaginary components,
also transforming the problem into the real domain. Janovsk\'a and Opfer
\cite{janovska+opfer2007_on,janovska+opfer2008_linear} characterised the
solution space and proposed numerical algorithms, and Mierzejewski
\cite{mierzejewski2011_linear} described the geometry of the solution
sets. 

These results are obtained in different ways. Tian
\cite{tian1999_similarity,tian2004_equations}, Shpakivskyi
\cite{shpakivsky2011_linear} and Turner \cite{turner2006_engineerin} pass to
a real representation of $S$, Janovsk\'a and Opfer
\cite{janovska+opfer2007_on,janovska+opfer2008_linear} and Shao, Li and
Huang \cite{shao+li++2020_basis-free} work within the quaternion algebra,
and Mierzejewski \cite{mierzejewski2011_linear} describes the solution sets
geometrically. Equation \eqref{eq:SILVESTER_EQUATION} is also the scalar
case of the quaternion matrix equation $AX-XB=C$, for which least-squares
and minimal-norm solutions are obtained from a real or complex
representation combined with the Moore--Penrose inverse, see for instance
\cite{zhang+wei++2015_special} and the references therein. In principle,
those results yield the least-squares solutions of
\eqref{eq:SILVESTER_EQUATION} as the outcome of a real $4\times 4$
computation. They do not, however, give them as explicit quaternionic
expressions, and they do not exhibit the structure of the singular case, in
which $\ker S$ is two-dimensional and its basis depends on the relative
position of $\ima$ and $\imb$. For the scalar equation itself, Tian
\cite[Remark~4.2]{tian2004_equations} left open which solution of the
singular inhomogeneous equation has minimal norm, and what the
least-squares solutions are when the equation is not solvable.

The main contribution of this work is a treatment of the singular
Sylvester equation that works directly with the function $S$ and uses
only quaternion algebra. For the homogeneous equation we show that the
solution space is the complex line $\qx[1]\C[\qb]$ through a nonzero
solution $\qx[1]$, with the explicit orthogonal basis
$\{\qx[1],\qx[1](\imb)\}$ in every case, including $\qb=\qa[][*]$, which is
excluded from Tian's two-parameter formula
\cite[eq.~(2.4)]{tian1999_similarity}
(Theorem~\ref{theorem:NONZERO_SOLUTIONS_OF_THE_HOMOGENEOUS_SYLVESTER_EQUATION}).
For the inhomogeneous equation we show that in the singular case $S$ is
skew-symmetric and satisfies $S^3=-4\Abs{\ima}^2S$
(Lemmas~\ref{lemma:S_IS_SKEW_SYMMETRIC} and~\ref{lemma:cubic}). From these
two facts and the description of $\ker S$ we obtain:
\begin{enumerate*}[label=(\roman*)]
\item the solvability of \eqref{eq:SILVESTER_EQUATION} as a Fredholm
alternative, $\qc\perp\ker S$, which is equivalent to Tian's condition
$\qa\qc=\qc\qb[][*]$ (Theorem~\ref{theorem:SOLVABILITY});
\item the Moore--Penrose inverse of $S$ in closed form,
$S^{+}=-S/(4\Abs{\ima}^2)$, together with the orthogonal projectors onto
$\ran S$ and $\ker S$ (Theorem~\ref{theorem:MP});
\item the minimum-norm solution of \eqref{eq:SILVESTER_EQUATION},
which is the first term of Tian's general solution, together with the
least-squares solutions for arbitrary $\qc$ (Theorem~\ref{theorem:general},
Corollary~\ref{cor:tian}), which answers the questions of
\cite[Remark~4.2]{tian2004_equations} for equation
\eqref{eq:SILVESTER_EQUATION}.
\end{enumerate*}

Our work is complementary to the matrix approach of Tian~\cite{tian2004_equations}, who represented $S$ as a real
$4\times 4$ matrix and determined its eigenvalues and a $\{1\}$-inverse.
We treat $S$ as a real-linear function on $\Hset$ with the inner
product~\eqref{eq:INNER_PRODUCT} and work within the quaternion algebra.
In this way we obtain the Moore--Penrose inverse of $S$, an orthogonal
basis of $\ker S$ that holds in every case, and the explicit
formula~\eqref{eq:square} for $S^2$, from which the orthogonal projectors
onto $\ran S$ and $\ker S$ as well as the minimum-norm and least-squares
solutions follow as explicit quaternionic expressions. Hence, the results
can be used directly in quaternion-based applications. We relate them to
the matrix representation of~\cite{tian2004_equations} in
Remarks~\ref{rem:matrix} and~\ref{remark:tian_matrix}.

By relating the solution space to quaternion square roots
\cite{niven1942_roots,tian1999_similarity}, we further obtain a geometric
interpretation of the solution space: for $\ima\ne\pm\imb$, the second basis vector
$\qx[1](\imb)$ is a real multiple of the square root of
$(\ima)(\imb[][*])$, which after normalisation is the unit
quaternion that rotates $\imb$ onto $\ima$ along the shortest arc
(Remark~\ref{remark:SQRT_INTERPRETATION}). The explicit basis is used in
\cite{radak+scheunert++2026_closed-form} to solve Wahba's problem
\cite{wahba1965_least} for two vector observations.

The rest of the article is organised as follows.
Section~\ref{sec:FUNDAMENTALS} collects the quaternion identities we use,
and Section~\ref{sec:QUATERNION_SQRT} recalls quaternion square roots.
Sections~\ref{sec:SOLUTION_OF_HOMOGENEOUS_SINGULAR_SYLVESTER_EQUATION}
and~\ref{sec:SOLUTION_OF_INHOMOGENEOUS_SINGULAR_SYLVESTER_EQUATION} treat
the homogeneous and the inhomogeneous singular Sylvester equation in quaternions,
respectively, and Section~\ref{sec:CONCLUSION} concludes our work.


\section{Quaternion Fundamentals}\label{sec:FUNDAMENTALS}

The real quaternion algebra is an associative division algebra over the field
of real numbers with basis $\{1,\qi,\qj,\qk\}$. The quaternion imaginary
units $\qi, \qj, \qk$ satisfy the relations $\qi\qj = -\qj\qi = \qk, 
\qj\qk = -\qk\qj = \qi, \qk\qi = -\qi\qk = \qj$, and 
$\qi[][2]\! = \qj[][2]\! = \qk[][2]\! = \qi\qj\qk = -1$. A quaternion
$\qa= \qa[0]+\qa[1]\qi+\qa[2]\qj+\qa[3]\qk$ can be written as
$\qa=\rea + \ima$, where $\rea=\qa[0]$ is the real scalar part of $\qa$, $\ima = \qa[1]\qi+\qa[2]\qj+\qa[3]\qk$ is the vector imaginary
part of $\qa$, and $\qa[0],\dots,\qa[3]$ are real numbers. The conjugate of a quaternion $\qa$ is defined as
$\qa[][*] = \rea - \ima$. The norm of a quaternion $\qa$ is defined as 
$\Abs{\qa} = \sqrt{\smash[b]{\qa\qa[][*]}} = \sqrt{\smash[b]{\qa[][*]\!\qa}}
= \sqrt{\smash[b]{\qa[0][2] + \qa[1][2] + \qa[2][2] + \qa[3][2]}}$. A quaternion
$\qa$ is called unit if $\Abs{\qa} = 1$ and it is called pure if $\rea=0$. The inverse of a nonzero quaternion $\qa$ is defined as $\qa[][-1] = \qa[][*] /\Abs{\qa}^2$.
Given two quaternions $\qa, \qb$, the quaternion product $\qa\qb$ is defined as
\begin{align*}
  \qa\qb = (\rea)(\reb) - \ima \! \cdot \imb + (\rea)(\imb) + (\reb)(\ima) + \ima \times \imb
\end{align*}
where the symbols $(\, \cdot \, , \times)$ represent the vector scalar and cross product, respectively.
The quaternion product satisfies the following properties
\begin{align*}
\Abs{\qa\qb} = \Abs{\qa}\kern0.1em\Abs{\qb}, \qquad
\Re{(\qa\qb)} = \Re{(\qb\qa)}, \qquad
(\qa\qb)^{*} = \qb[][*]\qa[][*].
\end{align*}
In particular, if $\qa, \qb, \qc$ are pure quaternions, then it holds
\begin{align*}
  \qa\qb &= - \qa\cdot\qb + \qa\times\qb, \qquad
  -\qa[][2] = \Abs{\qa}^2 , \qquad
  \qc \times (\qa \times \qb) = (\qc\cdot\qb)\qa - (\qc\cdot\qa)\qb.
\end{align*}
We regard $\Hset$ as a four-dimensional real vector space, that is,
$\Hset\cong\R[][4]$, with the Euclidean inner product
\begin{align}\label{eq:INNER_PRODUCT}
  \ip{\qa}{\qb}:=\Re\RBrace{\qa\qb[][*]}
  =(\rea)(\reb)+\ima\cdot\imb,
\end{align}
such that $\ip{\qa}{\qa}=\Abs{\qa}^2$.
Two quaternions $\qa,\qb$ are called orthogonal if $\ip{\qa}{\qb}=0$. For a
real subspace $U$ of $\Hset$ we write $\qa\perp U$ if $\ip{\qa}{\qb}=0$ for
all $\qb\in U$, and we denote by
$U^{\perp}=\{\qa\in\Hset\colon\qa\perp U\}$ the orthogonal complement of $U$.
The kernel and the range of the real-linear function $S$ given in
\eqref{eq:SYLVESTER_FUNCTION} are the real subspaces
$\ker S=\{\qx\in\Hset\colon S(\qx)=0\}$ and
$\ran S=\{S(\qx)\colon\qx\in\Hset\}$, respectively.
For a positive integer $m$ we denote by $S^{m}$ the $m$-fold composition of
$S$ with itself, that is, $S^{1}=S$ and
$S^{m}(\qx)=S\bigl(S^{m-1}(\qx)\bigr)$ for $m\ge2$. Further, we denote by
$S'$ the adjoint of $S$ with respect to the inner product
\eqref{eq:INNER_PRODUCT}, that is, the real-linear function satisfying
$\ip{S(\qx)}{\qy}=\ip{\qx}{S'(\qy)}$ for arbitrary quaternions $\qx,\qy$.

\begin{lemma}\label{lemma:AXIS_PARALLEL_QUATERNIONS}
For two nonreal quaternions $\qa, \qb$ it holds $\qa\qb = \qb\qa$ if and only if there
exists a real number $\lambda$ such that $\ima = \lambda(\imb)$.
\end{lemma}

\begin{proof}
Since $\Re{(\qa\qb)} = \Re{(\qb\qa)}$, it holds 
$\Re{(\qa\qb)} - \Re{(\qb\qa)} = 0$ such that
\begin{align*}
  \qa\qb - \qb\qa = \Im(\qa\qb) - \Im(\qb\qa) = \ima \times \imb - \imb \times \ima = 2\,(\ima\times\imb),
\end{align*}
which yields the assertion of the lemma.
\end{proof}

\begin{lemma}
\label{lemma:AXIS_PARALLEL_SAME_NORM_QUATERNIONS}
For two nonreal quaternions $\qa, \qb$ with $\rea = \reb$ and $\Abs{\qa} = \Abs{\qb}$ it
holds $\qa\qb = \qb\qa$ if and only if either $\ima = \imb$ or $\ima = -\imb$.
\end{lemma}

\begin{proof}
Given $(\rea)^2 - (\reb)^2 = 0$ and Lemma~\ref{lemma:AXIS_PARALLEL_QUATERNIONS} 
it holds $\qa\qb = \qb\qa$ if and only if there exists a real number $\lambda$
such that
\begin{align*}
  0 = \Abs{\qa}^2 - \Abs{\qb}^2 = (\imb)^2 - (\ima)^2 = (\imb)^2 - (\lambda \imb)^2 = (1-\lambda^2)(\imb)^2.
\end{align*}
Hence, $\qa\qb = \qb\qa$ holds if and only if $\lambda = \pm1$, which completes the proof.
\end{proof}

\begin{lemma}
\label{lemma:SUM_OF_QUATERNION_CROSS_PRODUCTS}
For pure quaternions $\qa, \qb$ with $\Abs{\qa} = \Abs{\qb}$ it holds
\begin{align}
  \label{eq:QUATERNION_CROSS_PRODUCT}
  \qa(\qa\times\qb) + (\qa\times\qb)\qb = [\qa\cdot(\qb - \qa)](\qa + \qb),
\end{align}
which equals zero if and only if $\qa\qb = \qb\qa$.
\end{lemma}

\begin{proof}
Since $\qa \cdot (\qa\times\qb) = 0$ and $\qa\cdot\qa = \Abs{\qa}^2 = \Abs{\qb}^2 = \qb\cdot\qb$, it holds 
\begin{align*}
  \qa (\qa\times\qb) + (\qa\times\qb)\qb &= \qa \times (\qa\times\qb) - \qb \times (\qa\times\qb)\\
  &= (\qa\cdot\qb)\qa - (\qa\cdot\qa)\qb - (\qb\cdot\qb)\qa + (\qb\cdot\qa)\qb\\
  &= [(\qa\cdot\qb) - (\qa\cdot\qa)] (\qa + \qb)\\
  &= [\qa\cdot(\qb - \qa)] (\qa + \qb).
\end{align*}
Further, by Lemma~\ref{lemma:AXIS_PARALLEL_SAME_NORM_QUATERNIONS} we have 
$\qa\qb = \qb\qa$ if and only if either $\qa = \qb$ or $\qa = -\qb$ and hence
if and only if \eqref{eq:QUATERNION_CROSS_PRODUCT} equals zero.
\end{proof}

\begin{lemma}\label{lemma:AXIS_PERPENDICULAR_QUATERNIONS}
For two nonreal quaternions $\qa, \qb$ it holds $\qa\qb = -\qb\qa$ if and only if
$\rea = \reb = 0$ and $\ima\cdot\imb=0$.
\end{lemma}
\begin{proof}
First assume that $\qa\qb = -\qb\qa$. Then we have $0 = \qa\qb + \qb\qa$ and in particular
	\begin{align*}
		0 = \Im(\qa\qb) + \Im(\qb\qa) = 2[(\rea)(\imb) + (\reb)(\ima)].
	\end{align*}
This implies $\rea=\reb=0$, since $\qa\qb\neq\qb\qa$
and by Lemma~\ref{lemma:AXIS_PARALLEL_QUATERNIONS} 
$\ima\neq\lambda(\imb)$ for all real numbers $\lambda$. Furthermore, we have
	\begin{align*}
		0=\Re(\qa\qb)+\Re(\qb\qa) = 2[(\rea)(\reb) - \ima\cdot\imb] = -2(\ima\cdot\imb).
	\end{align*}
Conversely, assume $\rea= \reb=0$ and $\ima\cdot\imb=0$. Then
	\begin{align*}
		\qa\qb + \qb\qa = \Im(\qa\qb) + \Im(\qb\qa) = \ima\times\imb + \imb\times\ima = 0,
	\end{align*}
which yields the assertion of the lemma.
\end{proof}

\begin{definition}\label{definition:SIMILAR_QUATERNIONS}
Two nonreal quaternions $\qa,\qb$ are called similar if there exists a nonzero
quaternion $\qp$ such that $\qp[][-1]\qa\qp = \qb$. In that case we write
$\qa \sim \qb$.
\end{definition}

\begin{remark}
In Definition \ref{definition:SIMILAR_QUATERNIONS}, similarity was restricted to nonreal quaternions. However, the concept extends naturally to the real case. For real similar quaternions $\qa,\qb$ it holds $\qa = \Re\RBrace{\qa\qp\qp[][-1]} = \Re\RBrace{\qp[][-1]\qa\qp} = \qb$ and further
	\begin{align*}
		\qp[][-1]\qa\qp = \qb \quad \Longleftrightarrow \quad \qa\qp - \qb\qp = 0 \quad \Longleftrightarrow \quad (\qa - \qb)\qp = 0,
	\end{align*}
such that every nonzero quaternion $\qp$ possesses this property. This implies that in the real case every quaternion is a solution of the homogeneous singular Sylvester equation. In contrast, the inhomogeneous singular Sylvester equation has no solution in the real case, as shown in Section~\ref{sec:SOLUTION_OF_INHOMOGENEOUS_SINGULAR_SYLVESTER_EQUATION}.

\end{remark}

\begin{lemma}\label{lemma:SIMILAR_QUATERNIONS_CHARACTERIZATION}
Two nonreal quaternions $\qa,\qb$ are similar if and only if $\rea = \reb$ and
$\Abs{\qa} = \Abs{\qb}$.
\end{lemma}

\begin{proof}
First assume that $\qa \sim \qb$. Then there exists a nonzero quaternion $\qp$
such that $\rea = \Re\RBrace{\qa\qp\qp[][-1]} = \Re\RBrace{\qp[][-1]\qa\qp} = \reb$.
Moreover, $\qa\qp = \qp\qb$ such that
\begin{align*}
  \Abs{\qa}\Abs{\qp} = \Abs{\qa\qp} = \Abs{\qp\qb} = \Abs{\qp}\Abs{\qb}
\end{align*}
and hence $\Abs{\qa} = \Abs{\qb}$. Conversely, assume that $\rea = \reb$ and 
$\Abs{\qa} = \Abs{\qb}$. Then it holds $\rea - \reb = 0$ and 
$\Abs{\qb}^2 - \Abs{\qa}^2 = 0$. Furthermore, we have
\begin{align}
  \qa\qp - \qp\qb = (\rea - \reb)\qp + (\ima)\qp - \qp(\imb) = (\ima)\qp - \qp(\imb).
  \label{eq:SIMILAR_QUATERNIONS_CHARACTERIZATION_PROOF_A}
\end{align}
Now, consider the following cases in order to find a proper nonzero quaternion
$\qp$. Case~1: Let $\qa\qb \ne \qb\qa$. Then by 
Lemma~\ref{lemma:AXIS_PARALLEL_QUATERNIONS} it holds $\ima \ne \lambda(\imb)$
for all real numbers $\lambda$. Hence, $\qp = \ima + \imb$ is nonzero.
Case~2: Let $\qa\qb = \qb\qa$. Then by 
Lemma~\ref{lemma:AXIS_PARALLEL_SAME_NORM_QUATERNIONS} it holds either
$\ima = \imb$ or $\ima = -\imb$.
Case~2a: Let $\ima = \imb$. Then 
$\qp = \ima + \imb = 2\,\imb$ is nonzero. 
Case~2b: Let $\ima = -\imb$. Since $\ima + \imb = 0$, 
choose an arbitrary nonreal quaternion $\qq$ with $(\ima)\qq \ne \qq(\ima)$.
Then $\qp = (\ima)\qq + \qq(\imb) = (\ima)\qq - \qq(\ima)$ is nonzero.
Given the nonzero quaternions $\qp$, for Cases~1 and~2a we have
\begin{align}
  (\ima)\qp - \qp(\imb) = (\ima)^2 - (\imb)^2 = \Abs{\qb}^2 - \Abs{\qa}^2 = 0
  \label{eq:SIMILAR_QUATERNIONS_CHARACTERIZATION_PROOF_B}
\end{align}
and for Case 2b further
\begin{align}
  (\ima)\qp - \qp(\imb) = (\ima)^2\qq - \qq(\imb)^2 = (\Abs{\qb}^2 - \Abs{\qa}^2)\qq = 0.
  \label{eq:SIMILAR_QUATERNIONS_CHARACTERIZATION_PROOF_C}
\end{align}
Finally, combining \eqref{eq:SIMILAR_QUATERNIONS_CHARACTERIZATION_PROOF_A},
\eqref{eq:SIMILAR_QUATERNIONS_CHARACTERIZATION_PROOF_B} as well as
\eqref{eq:SIMILAR_QUATERNIONS_CHARACTERIZATION_PROOF_A},
\eqref{eq:SIMILAR_QUATERNIONS_CHARACTERIZATION_PROOF_C} yields
$\qa\qp = \qp\qb$ with nonzero $\qp$, which completes the proof.
\end{proof}

\begin{corollary}\label{corollary:PURE_SIMILAR_QUATERNIONS_CHARACTERIZATION}
For pure quaternions $\qa,\qb$ it holds $\qa \sim \qb$ if and only if
$\Abs{\qa} = \Abs{\qb}$ and hence if and only if $\qa \sim -\qb$.
\end{corollary}

\begin{remark}\label{rem:SIMILAR_LITERATURE}
The result concerning the similarity of the quaternions $\qa$ and $\qb$,
as stated in Lemma~\ref{lemma:SIMILAR_QUATERNIONS_CHARACTERIZATION}, appears in \cite{brenner1951_matrices}, in
\cite[Theorem~2.2]{zhang1997_quaternion}, in
\cite[Theorem~2.1]{tian1999_similarity}, and as the equivalence between
(1) and (5) of \cite[Theorem~2.2.6]{rodman2014_topics}. The quaternion
$\qp=\ima+\imb$ of Cases~1 and~2a is also the one used in
\cite{zhang1997_quaternion,tian1999_similarity}. We have included the
proof because the case distinction, in particular Case~2b, is the one
that reappears in the description of the solution space of the
homogeneous singular Sylvester equation in Theorem~\ref{theorem:NONZERO_SOLUTIONS_OF_THE_HOMOGENEOUS_SYLVESTER_EQUATION}.
\end{remark}


\section{Quaternion square roots}\label{sec:QUATERNION_SQRT}

Square roots of quaternions are defined analogously to square roots of complex
numbers and are the solutions to the quadratic
equation~\eqref{eq:SQRT_DEFINING_EQUATION}, where $\qa$ represents a nonzero
quaternion.
\begin{align}
  \qx[][2] - \qa  = 0
  \label{eq:SQRT_DEFINING_EQUATION}
\end{align}
In Remark~\ref{remark:SQRT_INTERPRETATION} it is shown that one of the two
basis vectors of the solution space of the homogeneous singular Sylvester
equation is a quaternion square root. For this purpose we recall the
representation of the square roots in a form suited to that use.
Statements about both the existence and number of such solutions can be found in
\cite{niven1941_equations,niven1942_roots,eilenberg+niven1944_fundamenta,
sakkalis+ko++2019_roots}. According to this, \eqref{eq:SQRT_DEFINING_EQUATION} has
two solutions if and only if either $\qa$ is a nonreal quaternion or a positive
real number. Otherwise, if $\qa$ is a negative real number,
\eqref{eq:SQRT_DEFINING_EQUATION} has infinitely many solutions. The 
representation of these solutions will be derived next.

\begin{definition}
Let $\qa$ be a nonzero quaternion. The solutions $\qx$ of
\eqref{eq:SQRT_DEFINING_EQUATION} are called quaternion square roots of $\qa$ and
are denoted as $\textstyle\sqrt{\qa}$.
\end{definition}

\begin{lemma}
\label{lemma:QUATERNION_SQRT}
For a nonzero quaternion $\qa$ \eqref{eq:SQRT_DEFINING_EQUATION} has the
solutions
\begin{align}
  \qx = \pm \sqrt{\Abs{\qa}} \, \frac{\qp}{\Abs{\qp}},
  \label{eq:QUATERNION_SQRT}
\end{align}
where $\qp$ is either an arbitrary nonzero pure quaternion if $\qa$ is a negative real
number, or $\qp = \qa + \Abs{\qa}$ if $\qa$ is a nonreal quaternion or a positive
real number.
\end{lemma}

\begin{proof}
Given \eqref{eq:SQRT_DEFINING_EQUATION} we have $\qx\qx = \qa$ such that
$\Abs{\qx}^2 = \Abs{\qa}$ and $\Abs{\qx} = \sqrt{\Abs{\qa}}$. Now assume that
$\qa$ is nonreal. Then 
\begin{align}
  2(\rex)\qx = \qx\qx + \qx[][*]\qx = \qx\qx + \Abs{\qx}^2 = \qa + \Abs{\qa}.
  \label{eq:SQRT_RESULTING_EQUATION}
\end{align}
Assuming $\rex = 0$ we have $0 = \qa + \Abs{\qa}$, which contradicts that $\qa$
is nonreal. Hence, we must have $\rex \neq 0$ and further
\begin{align}
  \Abs{2\rex} = \frac{\Abs[\big]{\qa + \Abs{\qa}}}{\Abs{\qx}} 
  \quad\Longleftrightarrow\quad 
  2\rex = \pm\frac{\Abs[\big]{\qa + \Abs{\qa}}}{\sqrt{\Abs{\qa}}},
  \label{eq:REAL_PART_OF_QUATERNION_SQRT}
\end{align}
such that substituting \eqref{eq:REAL_PART_OF_QUATERNION_SQRT} into 
\eqref{eq:SQRT_RESULTING_EQUATION} yields \eqref{eq:QUATERNION_SQRT}.
Next, assume that $\qa$ is a real number. Then from 
\eqref{eq:SQRT_DEFINING_EQUATION} we have $0 = \Im(\qx\qx)$ and $a = \Re(\qx\qx)$
as well as
\begin{align*}
  0 = \Im(\qx\qx) = 2(\rex)(\imx) \quad\Longleftrightarrow\quad \rex=0 \text{~or~} \imx=0.
\end{align*}
Now assume that $\qa$ is positive. If $\rex=0$, then 
$a = \Re(\qx\qx) = -\Abs{\imx}^2$, which contradicts that $\qa$ is a positive
real number. Hence, we must have $\imx = 0$ and further 
$\qx = \rex = \pm\textstyle\sqrt{a}$. Since $\qa=\Abs{\qa}$, it holds 
$\qa + \Abs{\qa}=\Abs[\big]{\qa + \Abs{\qa}}$, which yields \eqref{eq:QUATERNION_SQRT}.
Finally, assume that $\qa$ is negative. If $\imx=0$, then 
$\qa = \Re(\qx\qx) = \Abs{\rex}^2$\!, which contradicts that $\qa$ is a negative
real number. Hence, we must have $\rex = 0$. Further, for every pure quaternion $\qp$
it holds $\qp\qp = -\Abs{\qp}^2$ such that
\begin{align*}
  \qx\qx = -\Abs{\qx}^2 = - \Abs{\qa} = \frac{\qp\qp}{\Abs{\qp}^2} \Abs{\qa},
\end{align*}
which yields \eqref{eq:QUATERNION_SQRT} and completes the proof.
\end{proof}

\begin{corollary}\label{corollary:UNIT_QUATERNION_SQRT}
The square roots of a nonzero quaternion $\qa$ are unit quaternions if and
only if quaternion $\qa$ is unit.
\end{corollary}

\begin{corollary}\label{corollary:QUATERNION_SQRT_IS_NOT_PURE}
The square roots of a nonreal quaternion $\qa$ have nonzero real parts, 
that is, the square roots of a nonreal quaternion $\qa$ are not pure.
\end{corollary}

\begin{corollary}\label{corollary:QUATERNION_SQRT_LINEAR_FORM}
For a nonreal quaternion $\qa$ there exist nonzero real numbers
$\lambda_0, \lambda_1$ such that
\vspace*{-1ex}
\begin{align*}
  \sqrt{\qa} = \pm(\lambda_0 + \lambda_1 \qa).
\end{align*}
In particular, we have
$\lambda_0 = \lambda_1 \Abs{\qa}$ and 
$\lambda_1 = \sqrt{\Abs{\qa}} / \Abs[\big]{\qa + \Abs{\qa}}$.
\end{corollary}

\begin{remark}
For nonreal $\qa$ the representation \eqref{eq:QUATERNION_SQRT} is known from Theorem~2.6 in \cite{tian1999_similarity}, and equivalent
forms are given in \cite{niven1942_roots} and in Theorem~2.5.2 in \cite{rodman2014_topics}. Lemma~\ref{lemma:QUATERNION_SQRT} merely states
the three cases (nonreal, positive real, negative real) in the single form \eqref{eq:QUATERNION_SQRT}, which is the form used in
Lemma~\ref{lemma:QUATERNION_PRODUCT_SQRT} and in
Remark~\ref{remark:SQRT_INTERPRETATION}.
\end{remark}

In order to relate Lemma~\ref{lemma:QUATERNION_SQRT} to the solutions of the
singular Sylvester equation, it is necessary to consider a special case of
\eqref{eq:SQRT_DEFINING_EQUATION}. For arbitrary quaternions $\qa,\qb$, for which the
product $\qa\qb$ is nonzero, the square roots of $\qa\qb$ are the solutions to
\begin{align}
  \qx[][2] - \qa\qb  = 0.
  \label{eq:PRODUCT_SQRT_DEFINING_EQUATION}
\end{align}
If the product $\qa\qb$ is nonreal, then
\eqref{eq:PRODUCT_SQRT_DEFINING_EQUATION} has exactly two solutions. In
addition, if $\Abs{\qa} = \Abs{\qb}$, the solutions of
\eqref{eq:PRODUCT_SQRT_DEFINING_EQUATION} have a special structure.

\begin{lemma}\label{lemma:QUATERNION_PRODUCT_SQRT}
For arbitrary quaternions $\qa, \qb$ with $\Abs{\qa} = \Abs{\qb}$, for which the
quaternion product $\qa\qb$ is nonreal, \eqref{eq:PRODUCT_SQRT_DEFINING_EQUATION} has the
solutions
\begin{align}
  \qx = \pm \frac{\qa(\qb + \qa[][*])}{\Abs{\qb + \qa[][*]}} = \pm \frac{(\qa+ \qb[][*])\qb}{\Abs{\qa+ \qb[][*]}}.
  \label{eq:SOLUTION_OF_PRODUCT_SQRT_DEFINING_EQUATION}
\end{align}
\end{lemma}

\begin{proof}
We have $\Abs{\qa\qb} = \Abs{\qa}^2 = \qa\qa[][*]$ such that 
$\sqrt{\Abs{\qa\qb}} = \Abs{\qa}$ and 
$\qa\qb + \Abs{\qa\qb} = \qa(\qb + \qa[][*])$. Then 
Lemma~\ref{lemma:QUATERNION_SQRT} with $\qa\qb$ instead of $\qa$ yields the
left-hand solutions of \eqref{eq:SOLUTION_OF_PRODUCT_SQRT_DEFINING_EQUATION}. 
The right-hand solutions of \eqref{eq:SOLUTION_OF_PRODUCT_SQRT_DEFINING_EQUATION}
also follow from Lemma~\ref{lemma:QUATERNION_SQRT}, since
$\Abs{\qa\qb} = \Abs{\qb}^2 = \qb[][*]\qb$ such that 
$\sqrt{\Abs{\qa\qb}} = \Abs{\qb}$ and 
$\qa\qb + \Abs{\qa\qb} = (\qa + \qb[][*])\qb$.
\end{proof}

\begin{corollary}\label{corollary:QUATERNION_PRODUCT_SQRT_LINEAR_FORM}
For two quaternions $\qa, \qb$ with $\Abs{\qa} = \Abs{\qb}$, for which the
quaternion product $\qa\qb$ is nonreal, there exist nonzero real numbers
$\lambda_0, \lambda_1$ such that
\begin{align*}
  \sqrt{\qa\qb} = \pm(\lambda_0 + \lambda_1 \qa\qb).
\end{align*}
In particular, we have
$\lambda_0 = \lambda_1 \Abs{\qa}\Abs{\qb}$ and 
$\lambda_1 = 1 / \Abs{\qa+ \qb[][*]} = 1 / \Abs{\qb + \qa[][*]}$.
\end{corollary}


\section{Solutions to the homogeneous singular Sylvester equation}
\label{sec:SOLUTION_OF_HOMOGENEOUS_SINGULAR_SYLVESTER_EQUATION}

We consider the real-linear Sylvester function \eqref{eq:SYLVESTER_FUNCTION},
which is called singular if there exists a nonzero
quaternion $\qx$ such that $S(\qx) = 0$. If $S$, and consequently the
quaternions $\qa, \qb$, satisfy this condition, then
\begin{align}
  S(\qx) = \qa\qx - \qx\qb = 0
  \label{eq:HOMOGENEOUS_SYLVESTER_EQUATION}
\end{align}
is referred to as a homogeneous singular Sylvester equation. In this section we
restate the existence and the general solution
to~\eqref{eq:HOMOGENEOUS_SYLVESTER_EQUATION} previously presented
in~\cite{tian1999_similarity} and provide short proofs within the quaternion
algebra. Moving beyond existing theory, we show in
Theorem~\ref{theorem:NONZERO_SOLUTIONS_OF_THE_HOMOGENEOUS_SYLVESTER_EQUATION}
that the solution space of~\eqref{eq:HOMOGENEOUS_SYLVESTER_EQUATION} is a
complex line through any nonzero solution, which yields an explicit
orthogonal basis of the solution space in every case, including the case
$\ima = -\imb$, which is excluded from the two-parameter formula
of~\cite[eq.~(2.4)]{tian1999_similarity}.

\begin{lemma}
\label{lemma:EXISTENCE_OF_SOLUTIONS_OF_THE_HOMOGENEOUS_SYLVESTER_EQUATION}
Let $\qa, \qb$ be nonreal quaternions. Then 
\eqref{eq:HOMOGENEOUS_SYLVESTER_EQUATION} has a nonzero solution if and only if
$\qa \sim \qb$.
\end{lemma}

\begin{proof}
If $\qx$ is a nonzero solution of \eqref{eq:HOMOGENEOUS_SYLVESTER_EQUATION},
then $\qx[][-1]\qa\qx = \qb$, that is, $\qa \sim \qb$ by
Definition~\ref{definition:SIMILAR_QUATERNIONS}.
Conversely, if $\qa \sim \qb$, then by Definition~\ref{definition:SIMILAR_QUATERNIONS} there
exists a nonzero $\qp$ such that
\begin{align*}
  \qp[][-1]\qa\qp = \qb \quad \Longleftrightarrow \quad
  \qa\qp = \qp\qb \quad \Longleftrightarrow \quad
  \qa\qp - \qp\qb = 0.
\end{align*}
Hence, the nonzero solution $\qx = \qp$ yields the assertion of the lemma.
\end{proof}

\begin{corollary}
\label{corollary:SINGULAR_GENERAL_LINEAR_FUNCTION}
Let $\qa, \qb$ be nonreal quaternions. Then the real-linear function 
$S(\qx) = \qa\qx - \qx\qb$ is singular if and only if $\qa \sim \qb$. In that
case it holds
\begin{align}
  S(\qx) = (\ima)\qx - \qx(\imb).
  \label{eq:SINGULAR_GENERAL_LINEAR_FUNCTION}
\end{align}
\end{corollary}

By \eqref{eq:SINGULAR_GENERAL_LINEAR_FUNCTION}, the quaternions $\qa, \qb$ may
be considered as pure in the proofs of the following statements, whereas the
statements themselves are presented in their general form.
In order to solve \eqref{eq:HOMOGENEOUS_SYLVESTER_EQUATION} for similar
quaternions $\qa$ and $\qb$, we shall proceed as in the proof of
Lemma~\ref{lemma:SIMILAR_QUATERNIONS_CHARACTERIZATION}. There, we distinguished
between the following cases. Case~1: $\qa\qb \neq \qb\qa$ and Case~2: $\qa\qb =
\qb\qa$. According to Lemma~\ref{lemma:AXIS_PARALLEL_SAME_NORM_QUATERNIONS}, the
latter holds if and only if either Case~2a: $\ima = \imb$ or Case~2b: 
$\ima =-\imb$. We show that in all cases the solution space is spanned by nonzero
solutions $\qx[1]$ and $\qx[1](\imb)$.

\begin{lemma}[{\cite[Theorem~2.1(b)]{tian1999_similarity}}]
\label{lemma:GENERAL_SOLUTION_OF_THE_HOMOGENEOUS_SYLVESTER_EQUATION}
Let $\qa,\qb$ be similar quaternions and $\qq$ an arbitrary quaternion. Then the general
solution of \eqref{eq:HOMOGENEOUS_SYLVESTER_EQUATION} is given as
\begin{align}
  \qx = (\ima)\qq + \qq(\imb).
  \label{eq:GENERAL_SOLUTION_OF_THE_HOMOGENEOUS_SYLVESTER_EQUATION}
\end{align}
\end{lemma}

\begin{proof}
By Corollary \ref{corollary:SINGULAR_GENERAL_LINEAR_FUNCTION}, the quaternions $\qa,\qb$ may, without loss of generality, be considered as pure. Let $\qy$ be an arbitrary solution of \eqref{eq:HOMOGENEOUS_SYLVESTER_EQUATION}.
Then it holds $\qa\qy = \qy\qb$. Define $\qq = \qa\qy + \qy\qb$. Then we have
\begin{align*}
  \qx = \qa\qq + \qq\qb = \qa[][2]\qy + 2\qa\qy\qb + \qy\qb[][2] = \lambda\qy,
\end{align*}
where $\lambda = -4\Abs{\qa}^2$. Hence, every solution of 
\eqref{eq:HOMOGENEOUS_SYLVESTER_EQUATION} can be obtained from
\eqref{eq:GENERAL_SOLUTION_OF_THE_HOMOGENEOUS_SYLVESTER_EQUATION} with
an appropriate $\qq$. Now, let $\qq$ be an arbitrary quaternion. Then, by
substituting \eqref{eq:GENERAL_SOLUTION_OF_THE_HOMOGENEOUS_SYLVESTER_EQUATION}
into \eqref{eq:HOMOGENEOUS_SYLVESTER_EQUATION}, we obtain
\begin{align*}
  \qa\qx - \qx\qb = \Abs{\qb}^2\qq + \qa\qq\qb - \qa\qq\qb - \Abs{\qa}^2\qq = 0.
\end{align*}
Hence, \eqref{eq:GENERAL_SOLUTION_OF_THE_HOMOGENEOUS_SYLVESTER_EQUATION} is
the general solution of~\eqref{eq:HOMOGENEOUS_SYLVESTER_EQUATION} for an arbitrary
quaternion $\qq$.
\end{proof}

\begin{theorem}
\label{theorem:NONZERO_SOLUTIONS_OF_THE_HOMOGENEOUS_SYLVESTER_EQUATION}
Let $\qa,\qb$ be similar quaternions and $\lambda, \mu$ arbitrary real numbers
with $\Abs{\lambda} + \Abs{\mu} \neq 0$. Let $\qx[1] = \Im(\qa + \qb)$ if
$\ima \ne -\imb$, and let $\qx[1]$ be an arbitrary nonzero pure quaternion
with $\ima\cdot\qx[1] = 0$ if $\ima = -\imb$. Then the nonzero solutions to
\eqref{eq:HOMOGENEOUS_SYLVESTER_EQUATION} are given as
\begin{align}
  \qx = \lambda\,\qx[1] + \mu\,\qx[1](\imb).
  \label{eq:NONZERO_SOLUTIONS_OF_THE_HOMOGENEOUS_SYLVESTER_EQUATION}
\end{align}
The two terms $\qx[1]$, $\qx[1](\imb)$ in \eqref{eq:NONZERO_SOLUTIONS_OF_THE_HOMOGENEOUS_SYLVESTER_EQUATION}
are orthogonal with respect to the inner product \eqref{eq:INNER_PRODUCT}. Hence, the solution space of \eqref{eq:HOMOGENEOUS_SYLVESTER_EQUATION} is the two-dimensional real subspace $\qx[1]\C[\qb]$, where
$\C[\qb] = \{\lambda + \mu\,\imb \colon \lambda,\mu\in\R\}$.
\end{theorem}

\begin{proof}
By Corollary \ref{corollary:SINGULAR_GENERAL_LINEAR_FUNCTION}, the quaternions $\qa,\qb$ may, without loss of generality, be considered as pure, so that $\Abs{\qa} = \Abs{\qb}$.
We first note that, with $\qx$, also $\qx\qb$ solves
\eqref{eq:HOMOGENEOUS_SYLVESTER_EQUATION}, since
$\qa(\qx\qb) = (\qa\qx)\qb = (\qx\qb)\qb$, and that $\qx[1]$, $\qx[1]\qb$
are nonzero by definition and orthogonal, since
$\ip{\qx[1]}{\qx[1]\qb} = \Re(\qx[1]\qb[][*]\qx[1][*]) = \Abs{\qx[1]}^2\kern1pt\reb = 0$. Hence, it remains
to show that every solution is of the form
\eqref{eq:NONZERO_SOLUTIONS_OF_THE_HOMOGENEOUS_SYLVESTER_EQUATION}.
We examine the following cases. Case~1: Let $\qa\qb\neq\qb\qa$. Then by Lemma~\ref{lemma:AXIS_PARALLEL_QUATERNIONS} it holds $\qa \ne \nu\qb$ for all real numbers $\nu$. Hence, for every quaternion $\qq$ there exist real numbers $\lambda_0, \dots, \lambda_3$ such that
\begin{align}
	\qq = \lambda_0 + \lambda_1\qa + \lambda_2\qb + \lambda_3(\qa\times\qb).
	\label{eq:GENERAL_QUATERNION_GIVEN_A_B}
\end{align}
Substituting \eqref{eq:GENERAL_QUATERNION_GIVEN_A_B} into~\eqref{eq:GENERAL_SOLUTION_OF_THE_HOMOGENEOUS_SYLVESTER_EQUATION}, we obtain the general solution of \eqref{eq:HOMOGENEOUS_SYLVESTER_EQUATION} as
\begin{align*}
	\qx = \lambda_0(\qa+\qb) - (\lambda_1 + \lambda_2) (\Abs{\qa}^2 - \qa\qb) + \lambda_3[\qa(\qa\times\qb) + (\qa\times\qb)\qb].
\end{align*}
By Lemma~\ref{lemma:SUM_OF_QUATERNION_CROSS_PRODUCTS} the last term is a real
multiple of $\qa + \qb = \qx[1]$, and
\begin{align*}
  \Abs{\qa}^2 - \qa\qb = -\qb[][2] - \qa\qb = -(\qa+\qb)\qb = -\qx[1]\qb,
\end{align*}
which yields~\eqref{eq:NONZERO_SOLUTIONS_OF_THE_HOMOGENEOUS_SYLVESTER_EQUATION}.
Case~2: Let $\qa\qb = \qb\qa$. Then by Lemma~\ref{lemma:AXIS_PARALLEL_SAME_NORM_QUATERNIONS} either $\qa = \qb$ or $\qa = -\qb$ holds.
Case~2a: Let $\qa = \qb$. Then by Lemma~\ref{lemma:GENERAL_SOLUTION_OF_THE_HOMOGENEOUS_SYLVESTER_EQUATION}, for an arbitrary quaternion $\qq$, the general solution of \eqref{eq:HOMOGENEOUS_SYLVESTER_EQUATION} is given as
	\begin{align*}
		\qx = \qa\qq + \qq\qa = 2\qa(\req) - 2\qa\cdot\imq.
	\end{align*}
Then, with $\qx[1] = \qa + \qb = 2\qa$, $\qx[1]\qb = 2\qa[][2] = -2\Abs{\qa}^2$, and real numbers $\lambda = \req$ and $\mu = (\qa\cdot\imq)/\Abs{\qa}^2$, it follows
\begin{align*}
	2\qa(\req) = \lambda\,\qx[1], \qquad -2\qa\cdot\imq = \mu\,\qx[1]\qb,
\end{align*}
which yields \eqref{eq:NONZERO_SOLUTIONS_OF_THE_HOMOGENEOUS_SYLVESTER_EQUATION}. 
Case~2b: Let $\qa=-\qb$. Then \eqref{eq:HOMOGENEOUS_SYLVESTER_EQUATION} reads
$\qa\qx = -\qx\qa$. A nonzero real number $\qx$ is not a solution, since
$\qa\qx + \qx\qa = 2\qx\qa \ne 0$. For a nonreal quaternion $\qx$,
Lemma~\ref{lemma:AXIS_PERPENDICULAR_QUATERNIONS} yields $\qa\qx = -\qx\qa$ if
and only if $\rex = 0$ and $\qa\cdot\imx = 0$, since $\qa$ is pure.
Hence, the solutions are exactly the pure quaternions $\qx$ with
$\qa\cdot\qx = 0$, which form a two-dimensional real subspace. In particular,
every nonzero pure quaternion $\qx[1]$ with $\qa\cdot\qx[1] = 0$ is a
solution. As noted at the beginning of the proof, $\qx[1]\qb$ is then a
nonzero solution orthogonal to $\qx[1]$. Thus, $\qx[1]$ and $\qx[1]\qb$ span
the solution space of~\eqref{eq:HOMOGENEOUS_SYLVESTER_EQUATION}, which yields
\eqref{eq:NONZERO_SOLUTIONS_OF_THE_HOMOGENEOUS_SYLVESTER_EQUATION}.
\end{proof}

\begin{remark}\label{remark:SQRT_INTERPRETATION}
The representation in \eqref{eq:NONZERO_SOLUTIONS_OF_THE_HOMOGENEOUS_SYLVESTER_EQUATION} identifies the quaternionic directions that span the solution space of~\eqref{eq:HOMOGENEOUS_SYLVESTER_EQUATION}. Specifically, the solution is decomposed into two orthogonal components aligned with the sum $\qx[1] = \Im(\qa+\qb)$ and with
\begin{align*}
  \qx[1](\imb) = -\bigl(\Abs{\ima}\Abs{\imb} - (\ima)(\imb)\bigr)
  = \pm\Abs{\ima+\imb}\sqrt{(\ima)(\imb[][*])},
\end{align*}
where the first equality uses $(\imb)^2 = -\Abs{\imb}^2 = -\Abs{\ima}\Abs{\imb}$ and the last equality follows from Corollary~\ref{corollary:QUATERNION_PRODUCT_SQRT_LINEAR_FORM}. Up to normalisation these are the two particular solutions of \cite[eq.~(2.4)]{tian1999_similarity}. Theorem~\ref{theorem:NONZERO_SOLUTIONS_OF_THE_HOMOGENEOUS_SYLVESTER_EQUATION} adds that they are orthogonal, that they span the solution space, and what replaces them if $\ima = -\imb$. The quaternion square root $\sqrt{(\ima)(\imb[][*])}$, normalised, is the unit quaternion that rotates $\imb$ onto $\ima$ about the axis $\ima\times\imb$, while $\qx[1]/\Abs{\qx[1]}$ is the half-turn about the bisector of $\ima$ and $\imb$.
This explicit basis is used in~\cite{radak+scheunert++2026_closed-form} to
solve Wahba's problem~\cite{wahba1965_least},
\cite[Chapter~5]{markley+crassidis2014_fundamenta} for two vector
observations, where the attitude quaternion is a common solution of two
homogeneous singular Sylvester equations.
\end{remark}


\section{Solutions to the inhomogeneous singular Sylvester equation}
\label{sec:SOLUTION_OF_INHOMOGENEOUS_SINGULAR_SYLVESTER_EQUATION}

In this section we present the condition for solvability and the general
solution of the inhomogeneous singular Sylvester equation, and we determine
its minimum-norm solution as well as its least-squares solutions when it is
not solvable. For similar
quaternions $\qa,\qb$ and a nonzero quaternion $\qc$, the inhomogeneous
singular Sylvester equation is given as
\begin{align}
  S(\qx) = \qa\qx-\qx\qb = \qc.
  \label{eq:INHOMOGENEOUS_SINGULAR_SYLVESTER_EQUATION}
\end{align}
If $\qa,\qb$ are real numbers, \eqref{eq:INHOMOGENEOUS_SINGULAR_SYLVESTER_EQUATION} has no solution, since $(\qa-\qb)\qx=0$ for all quaternions $\qx$, and every quaternion is a least-squares solution with residual $\Abs{\qc}$.
Throughout this section, $\qx[1]$ is as given in Theorem~\ref{theorem:NONZERO_SOLUTIONS_OF_THE_HOMOGENEOUS_SYLVESTER_EQUATION}, so that $\{\qx[1],\qx[1](\imb)\}$ is an orthogonal basis of $\ker S$.


\begin{lemma}[{\cite[Theorem~2.5]{tian2004_equations}}]\label{lemma:SOLVABILITY_OF_THE_INHOMOGENEOUS_SINGULAR_SYLVESTER_EQUATION}
Let $\qa,\qb$ be similar quaternions and $\qc$ a nonzero quaternion. Then \eqref{eq:INHOMOGENEOUS_SINGULAR_SYLVESTER_EQUATION} has a solution if and only if $\qa\qc =\qc\qb[][*]$.
\end{lemma}

\begin{proof}
By Corollary~\ref{corollary:SINGULAR_GENERAL_LINEAR_FUNCTION}, we may assume that $\qa,\qb$ are pure quaternions, so that $\qc\qb[][*] = -\qc\qb$. 
Since $\Abs{\qa} = \Abs{\qb}$, for every solution $\qx$ we have
\begin{align*}
  \qa\qc = \qa(\qa\qx-\qx\qb) = -\Abs{\qa}^2\qx - \qa\qx\qb =  -\qa\qx\qb - \qx\Abs{\qb}^2 = -(\qa\qx-\qx\qb)\qb = -\qc\qb.
\end{align*}
Conversely, if $\qa\qc=-\qc\qb$, then $\qx = -\qa\qc/(2\Abs{\qa}^2)$ is a solution of~\eqref{eq:INHOMOGENEOUS_SINGULAR_SYLVESTER_EQUATION}, since
\begin{align*}
  \qa\qx - \qx\qb = \frac{\qc}{2} + \frac{\qa\qc\qb}{2\Abs{\qa}^2} = \frac{\qc}{2} - \frac{\qc\qb[][2]}{2\Abs{\qa}^2} = \qc,
\end{align*}
which completes the proof.
\end{proof}

\begin{lemma}[{\cite[Theorem~2.5]{tian2004_equations}}]\label{lemma:SOLUTION_OF_THE_INHOMOGENEOUS_SINGULAR_SYLVESTER_EQUATION}
Let $\qa,\qb$ be similar quaternions and $\qc$ a nonzero quaternion with $\qa\qc =\qc\qb[][*]$. Then, for an arbitrary quaternion $\qq$, the general solution of~\eqref{eq:INHOMOGENEOUS_SINGULAR_SYLVESTER_EQUATION} is given as
\begin{align}
  \qx = (\ima)\frac{\qq - \qc}{4\Abs{\ima}^2} + \frac{\qq + \qc}{4\Abs{\imb}^2}(\imb).
  \label{eq:GENERAL_SOLUTION_OF_THE_INHOMOGENEOUS_SYLVESTER_EQUATION}
\end{align}
\end{lemma}

\begin{proof}
By Corollary~\ref{corollary:SINGULAR_GENERAL_LINEAR_FUNCTION}, the quaternions $\qa,\qb$ may, without loss of generality, be considered as pure. Hence, $\qa\qc=-\qc\qb$. Substituting $\qc=\qa\qa[][-1]\qc$ and $\qc = \qc\qb[][-1]\qb$ into \eqref{eq:INHOMOGENEOUS_SINGULAR_SYLVESTER_EQUATION} yields
\begin{align}
  0 = \qa\qx - \qx\qb - \frac{1}{2}(\qa\qa[][-1]\qc + \qc\qb[][-1]\qb) = \qa \Big(\qx + \frac{\qa\qc}{2\Abs{\qa}^2}\Big) - \Big(\qx - \frac{\qc\qb}{2\Abs{\qb}^2}\Big)\qb.
  \label{eq:SINGULAR_SYLVESTER_EQUATION_PROOF_SOLUTION_INHOMOGENEOUS}
\end{align}
Since $\qa\qc=-\qc\qb$, \eqref{eq:SINGULAR_SYLVESTER_EQUATION_PROOF_SOLUTION_INHOMOGENEOUS} is a homogeneous singular Sylvester equation, and by Lemma~\ref{lemma:GENERAL_SOLUTION_OF_THE_HOMOGENEOUS_SYLVESTER_EQUATION} we have
\begin{align}
  \qx + \frac{\qa\qc}{2\Abs{\qa}^2} = \qa\qq + \qq\qb \qquad \Longleftrightarrow \qquad \qx = \qa\qq + \qq\qb - \frac{\qa\qc}{2\Abs{\qa}^2},
  \label{eq:QX_PROOF_SOLUTION_INHOMOGENEOUS}
\end{align}
for an arbitrary quaternion $\qq$.
By substituting $\qa\qc = \tfrac12(\qa\qc - \qc\qb)$ into \eqref{eq:QX_PROOF_SOLUTION_INHOMOGENEOUS}, we obtain \eqref{eq:GENERAL_SOLUTION_OF_THE_INHOMOGENEOUS_SYLVESTER_EQUATION}, which yields the assertion of the lemma.
\end{proof}

The general solution of~\eqref{eq:INHOMOGENEOUS_SINGULAR_SYLVESTER_EQUATION} was given by Tian \cite[eq.~(2.16)]{tian2004_equations} in the form
\begin{align}\label{eq:TIAN_INHOMOGENEOUS_SOLUTION}
  \qx=\frac{\qc\qb-\qa\qc}{4\Abs{\ima}^2}
      +\qp-\frac{1}{\Abs{\ima}^2}\,(\ima)\qp(\imb),
\end{align}
where $\qp$ is an arbitrary quaternion. The representations
\eqref{eq:GENERAL_SOLUTION_OF_THE_INHOMOGENEOUS_SYLVESTER_EQUATION} and
\eqref{eq:TIAN_INHOMOGENEOUS_SOLUTION} have the
same particular solution and differ only in the parametrisation of the
homogeneous part. For arbitrary quaternions $\qq$ and $\qp$, both
$(\ima)\qq+\qq(\imb)$ and $\qp-(\ima)\qp(\imb)/\Abs{\ima}^2$ range over
$\ker S$, the former by
Lemma~\ref{lemma:GENERAL_SOLUTION_OF_THE_HOMOGENEOUS_SYLVESTER_EQUATION} and
the latter by Corollary~\ref{cor:tian} below such that
\eqref{eq:GENERAL_SOLUTION_OF_THE_INHOMOGENEOUS_SYLVESTER_EQUATION} and
\eqref{eq:TIAN_INHOMOGENEOUS_SOLUTION} describe
the same set of solutions.
The norm of
\eqref{eq:TIAN_INHOMOGENEOUS_SOLUTION} depends on $\qp$. In the remainder of
this section we determine the solution of minimal norm and, if
\eqref{eq:INHOMOGENEOUS_SINGULAR_SYLVESTER_EQUATION} is not solvable, its
least-squares solutions, which were asked for in
\cite[Remark~4.2]{tian2004_equations}.


\begin{lemma}\label{lemma:S_IS_SKEW_SYMMETRIC}
The singular Sylvester function \eqref{eq:SYLVESTER_FUNCTION} is skew-symmetric, that is, for
arbitrary quaternions $\qx$, $\qy$ it holds
\begin{align*}
  \ip{S(\qx)}{\qy}=-\ip{\qx}{S(\qy)}.
\end{align*}
\end{lemma}

\begin{proof}
By Corollary~\ref{corollary:SINGULAR_GENERAL_LINEAR_FUNCTION}, the
quaternions $\qa,\qb$ may, without loss of generality, be considered as
pure. Then for arbitrary quaternions $\qx$, $\qy$ we have
\begin{align*}
  \ip{S(\qx)}{\qy}+\ip{\qx}{S(\qy)}
  = \Re\RBrace{\qa\qx\qy[][*]-\qx\qb\qy[][*]}
   + \Re\RBrace{\qx\qy[][*]\qa[][*]-\qx\qb[][*]\qy[][*]} = 0,
\end{align*}
which yields the assertion of the lemma.
\end{proof}

\begin{corollary}
\label{corollary:H_IS_DIRECT_ORTHOGONAL_SUM}
The adjoint of the singular Sylvester function \eqref{eq:SYLVESTER_FUNCTION}
is given as $S' = -S$. Then, for arbitrary quaternions $\qx$, $\qy$ such that 
$\qy \in \ker S'$ it holds $0 = \ip{\qx}{S'(\qy)} = \ip{S(\qx)}{\qy}$, that is, 
$\qy \in (\ran S)^{\perp}$ and hence $(\ran S)^{\perp} = \ker S'$. Since $\Hset$
is finite-dimensional, $\ran S$ is closed, which yields 
\begin{align*}
  \ran S = (\ran S)^{\perp\perp} = (\ker S')^{\perp} = (\ker -S)^{\perp} = (\ker S)^{\perp},
\end{align*}
such that $\Hset=\ker S \oplus \ran S$ is an orthogonal direct sum.
\end{corollary}

\begin{lemma}\label{lemma:cubic}
The singular Sylvester function \eqref{eq:SYLVESTER_FUNCTION} satisfies
\begin{align}\label{eq:cubic}
  S^{3}=-4\Abs{\ima}^2S .
\end{align}
\end{lemma}

\begin{proof}
By Corollary~\ref{corollary:SINGULAR_GENERAL_LINEAR_FUNCTION}, the quaternions
$\qa,\qb$ may, without loss of generality, be considered as pure. Then for an
arbitrary quaternion $\qx$ we have
\begin{align}
  S^2(\qx)=\qa(\qa\qx-\qx\qb)-(\qa\qx-\qx\qb)\qb
          =\qa[][2]\qx-2\qa\qx\qb+\qx\qb[][2]
          =-2\Abs{\qa}^2\qx-2\qa\qx\qb
  \label{eq:square}
\end{align}
and further
\vspace*{-1ex}
\begin{align*}
  S(\qa\qx\qb)=\qa[][2]\qx\qb-\qa\qx\qb[][2]
              =\Abs{\qa}^2(\qa\qx-\qx\qb)
              =\Abs{\qa}^2S(\qx) .
\end{align*}
Since $S$ is real-linear, applying $S$ to \eqref{eq:square} yields
\begin{align*}
  S^3(\qx)=-2\Abs{\qa}^2S(\qx)-2S(\qa\qx\qb)=-4\Abs{\qa}^2S(\qx),
\end{align*}
which completes the proof.
\end{proof}

\begin{remark}\label{rem:matrix}
Lemmas~\ref{lemma:S_IS_SKEW_SYMMETRIC} and~\ref{lemma:cubic} can also be
obtained from the matrix representation of $S$. Let $\delta$ be the matrix of
$S$ with respect to the basis $\{1,\qi,\qj,\qk\}$. Then $\delta$ is skew-symmetric
and therefore normal. By \cite[Lemma~2.1(b)]{tian2004_equations} its
eigenvalues are $(0,0,\pm 2\sqrt{-1}\,\Abs{\ima})$. Hence, the minimal
polynomial of $\delta$ is $t\,(t^2+4\Abs{\ima}^2)$, which is
exactly~\eqref{eq:cubic}. Our proofs avoid the matrix representation. In
addition, they provide the explicit formula~\eqref{eq:square} for $S^2$.
This formula yields the closed quaternionic form of the projectors in 
Theorem~\ref{theorem:MP} and of the
solutions in Theorem~\ref{theorem:general}.
\end{remark}


\begin{theorem}\label{theorem:SOLVABILITY}
Let $\qc$ be a nonzero quaternion. Then for the singular Sylvester function
\eqref{eq:SYLVESTER_FUNCTION} the following statements are equivalent.
\begin{enumerate}[label=(\roman*)]
\item The inhomogeneous singular Sylvester equation
\eqref{eq:INHOMOGENEOUS_SINGULAR_SYLVESTER_EQUATION} has a solution.
\item $\qa\qc=\qc\qb[][*]$.
\item $\qc\perp\ker S$, that is,
$\ip{\qc}{\qx[1]}=\ip{\qc}{\qx[1](\imb)}=0$ with $\qx[1]$ as in
Theorem~\ref{theorem:NONZERO_SOLUTIONS_OF_THE_HOMOGENEOUS_SYLVESTER_EQUATION}.
\item $S^{2}(\qc)=-4\Abs{\ima}^2\qc$.
\end{enumerate}
\end{theorem}

\begin{proof}
The statements (i) and (ii) are equivalent by Lemma~\ref{lemma:SOLVABILITY_OF_THE_INHOMOGENEOUS_SINGULAR_SYLVESTER_EQUATION}.
Further, statement (i) holds if and only if $\qc\in\ran S$ and by Corollary~\ref{corollary:H_IS_DIRECT_ORTHOGONAL_SUM} if and only if $\qc\in(\ker S)^{\perp}$, that is, $\qc \perp\ker S$. Since by
Theorem~\ref{theorem:NONZERO_SOLUTIONS_OF_THE_HOMOGENEOUS_SYLVESTER_EQUATION}
the quaternions $\qx[1]$ and $\qx[1](\imb)$ form a basis of $\ker S$, it holds
$\qc\perp\ker S$ if and only if
$\ip{\qc}{\qx[1]} = \ip{\qc}{\qx[1](\imb)} = 0$. Hence, the statements (i) and
(iii) are equivalent. Next, assume that statement (i) holds. Then there
exists a quaternion $\qy$ with $\qc = S(\qy)$ such that by
Lemma~\ref{lemma:cubic} we have
\begin{align*}
  S^{2}(\qc) = S^{3}(\qy) = -4\Abs{\ima}^2S(\qy) = -4\Abs{\ima}^2\qc,
\end{align*}
which is statement (iv). Conversely, assume that statement (iv) holds. Since
$\qa$ is nonreal, it holds $\Abs{\ima}\ne0$ and we may define
$\qx = -S(\qc)/(4\Abs{\ima}^2)$. Since $S$ is real-linear, it holds
\begin{align*}
  S(\qx) = -\frac{1}{4\Abs{\ima}^2}\,S^{2}(\qc) = \qc,
\end{align*}
such that $\qx$ is a solution of
\eqref{eq:INHOMOGENEOUS_SINGULAR_SYLVESTER_EQUATION}, which completes the proof.
\end{proof}

\begin{remark}
Note that Theorem~\ref{theorem:SOLVABILITY} is stated for a nonzero quaternion $\qc$.
However, for $\qc = 0$, the statements (i)--(iv) hold trivially. Hence, 
the quaternions $\qc$ satisfying (i)--(iv) form the subspace $\ran S$ of $\Hset$. Theorem~\ref{theorem:SOLVABILITY} is the singular case of the Fredholm alternative for the real-linear function $S$ on $\Hset\cong\R[][4]$. By this alternative,
either $\ker S = \{0\}$ and
\eqref{eq:INHOMOGENEOUS_SINGULAR_SYLVESTER_EQUATION} has a unique solution for
every quaternion $\qc$, which is the regular case \cite{johnson1944_on}, or
$\ker S \ne \{0\}$ and \eqref{eq:INHOMOGENEOUS_SINGULAR_SYLVESTER_EQUATION} is
solvable if and only if $\qc$ is orthogonal to $\ker S' $. Since $S$ is
skew-symmetric by Lemma~\ref{lemma:S_IS_SKEW_SYMMETRIC}, it holds $\ker S'  = \ker S$ such
that statement (iii) reduces solvability to the vanishing of the two real
inner products of $\qc$ with the basis vectors $\qx[1]$ and $\qx[1](\imb)$ of
$\ker S$.

Statement (ii) is Tian's condition \cite[Theorem~2.5]{tian2004_equations}.
Hence, the quaternions $\qc$ for which~\eqref{eq:INHOMOGENEOUS_SINGULAR_SYLVESTER_EQUATION} is solvable are exactly
the solutions of the homogeneous singular Sylvester equation 
$\qa\qc - \qc\qb[][*]=0$. Since $\qa \sim \qb[][*]$,
Theorem~\ref{theorem:NONZERO_SOLUTIONS_OF_THE_HOMOGENEOUS_SYLVESTER_EQUATION}
yields $\ran S = \qc[1]\C[\qb]$ with the orthogonal basis
$\{\qc[1], \qc[1](\imb)\}$, where $\qc[1] = \Im(\qa - \qb)$ if
$\ima \ne \imb$, and $\qc[1]$ is an arbitrary nonzero pure quaternion with
$\ima\cdot\qc[1] = 0$ if $\ima = \imb$.

By statement (iv), $\ran S$ is the eigenspace of $S^{2}$ for the eigenvalue
$-4\Abs{\ima}^2$. Since $S^{2}(\qx)=0$ for every $\qx\in\ker S$ and
$\Hset = \ker S\oplus\ran S$ by
Corollary~\ref{corollary:H_IS_DIRECT_ORTHOGONAL_SUM}, the only other
eigenvalue of $S^{2}$ is $0$ with the eigenspace $\ker S$. Accordingly,
$S$ has the double eigenvalue $0$ and the simple complex eigenvalues
$\pm2\sqrt{-1}\,\Abs{\ima}$, in agreement with
\cite[Lemma~2.1(b)]{tian2004_equations}.
\end{remark}


For the singular Sylvester function $S$ there exists a nonzero quaternion $\qx$ with
$S(\qx) = 0$. Hence, $S$ is not injective and has no inverse. However, every real-linear function
$S$ of a quaternion variable possesses a unique Moore--Penrose inverse $S^{+}$,
which is the real-linear function satisfying the four Penrose equations~\cite[Chapter~1]{BenIsrael2003}
\begin{align}
  S S^{+}\kern-1pt S = S, \qquad S^{+}\kern-1pt S S^{+} = S^{+}\kern-1pt, \qquad
  (S S^{+}\kern-1pt)'  = S S^{+}\kern-1pt, \qquad (S^{+}\kern-1pt S)'  = S^{+}\kern-1pt S.
  \label{eq:PENROSE_EQUATIONS}
\end{align}
We show that the Moore--Penrose inverse of the singular Sylvester function $S$
is a real multiple of $S$. Furthermore, we provide the orthogonal
projectors onto $\ran S$ and $\ker S$ in closed form.

\begin{theorem}\label{theorem:MP}
The Moore--Penrose inverse of the singular Sylvester function
\eqref{eq:SYLVESTER_FUNCTION} is given as
\begin{align}
  S^{+} = -\frac{1}{4\Abs{\ima}^2}\,S .
  \label{eq:MOORE_PENROSE_OF_SINGULAR_SYLVESTER_FUNCTION}
\end{align}
Moreover, the orthogonal projectors $P=SS^{+}\kern-1pt = S^{+}\kern-1pt S$ onto $\ran S$ and
$I-P$ onto $\ker S$ for an arbitrary quaternion $\qx$ are given as
\begin{align}
  P(\qx) = \frac12\Bigl(\qx+\frac{(\ima)\qx(\imb)}{\Abs{\ima}^2}\Bigr),
  \qquad
  (I-P)(x) = \frac12\Bigl(\qx-\frac{(\ima)\qx(\imb)}{\Abs{\ima}^2}\Bigr).
  \label{eq:PROJECTORS_TO_KER_AND_RAN_OF_SINGULAR_SYLVESTER_FUNCTION}
\end{align}
\end{theorem}

\begin{proof}
By Corollary~\ref{corollary:SINGULAR_GENERAL_LINEAR_FUNCTION}, the
quaternions $\qa,\qb$ may, without loss of generality, be considered as
pure. Let $M = -S/(4\Abs{\qa}^2)$. Then by
Lemma~\ref{lemma:cubic} it holds
\begin{align*}
  SMS = -\frac{1}{4\Abs{\qa}^2}\,S^{3} = S, \qquad
  MSM = \frac{1}{16\Abs{\qa}^4}\,S^{3}
              = -\frac{1}{4\Abs{\qa}^2}\,S = M.
\end{align*}
Since $M$ is a real multiple of $S$, we also have 
$SM = MS = -S^{2}/(4\Abs{\qa}^2)$. 
Further, by Lemma~\ref{lemma:S_IS_SKEW_SYMMETRIC} it holds $S' = -S$ such that
\begin{align*}
  (SM)'  = -\frac{1}{4\Abs{\qa}^2}\,(S' )^{2}
               = -\frac{1}{4\Abs{\qa}^2}\,S^{2} = SM
\end{align*}
and hence $(MS)'  = MS$. Thus, $M$
satisfies \eqref{eq:PENROSE_EQUATIONS}, that is, $S$ possesses the unique
Moore--Penrose inverse $S^{+}\kern-1pt = M$ as given in \eqref{eq:MOORE_PENROSE_OF_SINGULAR_SYLVESTER_FUNCTION}.

Now, let $P = SS^{+}\kern-1pt$. Then $P$ is symmetric by \eqref{eq:PENROSE_EQUATIONS} and idempotent by
$P^{2} = S(S^{+}\kern-1pt SS^{+}\kern-1pt) = SS^{+}\kern-1pt = P$.
Hence, $P$
is an orthogonal projector. For an arbitrary quaternion $\qx$ we further have
$P(\qx) = S\bigl(S^{+}\kern-1pt(\qx)\bigr) \in \ran S$ and
$P\bigl(S(\qx)\bigr) = SS^{+}\kern-1pt S(\qx) = S(\qx)$ such that the range of $P$
is $\ran S$. Then $I - P$ is the orthogonal projector onto $(\ran S)^{\perp} = \ker S$
by Corollary~\ref{corollary:H_IS_DIRECT_ORTHOGONAL_SUM}. Finally, 
substituting \eqref{eq:square} into $P(x) = (SS^{+}\kern-1pt)(x) = -S^{2}(x)/(4\Abs{\qa}^2)$  yields
the explicit representation of the orthogonal projectors 
\eqref{eq:PROJECTORS_TO_KER_AND_RAN_OF_SINGULAR_SYLVESTER_FUNCTION}, which
completes the proof.
\end{proof}

\begin{remark}\label{remark:tian_matrix}
In \cite[Lemma~2.1(d)]{tian2004_equations} Tian noted that
$-\delta/(4\Abs{\ima}^2)$ is a $\{1\}$-inverse of the real matrix
representation $\delta$ of $S$, that is, it satisfies only the first of
the Penrose equations \eqref{eq:PENROSE_EQUATIONS}.
Theorem~\ref{theorem:MP} shows that it satisfies all four and hence is the
Moore--Penrose inverse. This property is required for the minimum-norm
and least-squares statements in Theorem~\ref{theorem:general}. A $\{1\}$-inverse yields a solution of
\eqref{eq:INHOMOGENEOUS_SINGULAR_SYLVESTER_EQUATION}, provided that a
solution exists. However, it yields in general neither the minimum-norm
solution nor the least-squares solutions if
\eqref{eq:INHOMOGENEOUS_SINGULAR_SYLVESTER_EQUATION} is not solvable,
whereas the Moore--Penrose inverse yields both.
\end{remark}


If $\qc\notin\ran S$, then
\eqref{eq:INHOMOGENEOUS_SINGULAR_SYLVESTER_EQUATION} has no solution. In
that case we seek the quaternions $\qx$ that minimise the residual $\Abs{S(x)-\qc}$,
which are referred to as the least-squares solutions of
\eqref{eq:INHOMOGENEOUS_SINGULAR_SYLVESTER_EQUATION}. If
\eqref{eq:INHOMOGENEOUS_SINGULAR_SYLVESTER_EQUATION} has a solution, then
the minimal residual is zero and the least-squares solutions
are exactly the solutions. Hence, both cases are treated together in the
following theorem.

\begin{theorem}\label{theorem:general}
Let $\qc$ be a nonzero quaternion and $\lambda,\mu$ arbitrary real numbers.
Further, let $\qx[1]$ be as in
Theorem~\ref{theorem:NONZERO_SOLUTIONS_OF_THE_HOMOGENEOUS_SYLVESTER_EQUATION},
$P$ as in Theorem~\ref{theorem:MP}, and
\begin{align}\label{eq:xc}
  \qx[c] = S^{+}\kern-1pt(\qc) = -\frac{1}{4\Abs{\ima}^{2}}S(\qc),
\end{align}
regarding the singular Sylvester function \eqref{eq:SYLVESTER_FUNCTION}. Then
the following statements hold.
\begin{enumerate}[label=(\roman*)]
\item If \eqref{eq:INHOMOGENEOUS_SINGULAR_SYLVESTER_EQUATION} has no solution,
the quaternions $\qx$ minimising $\Abs{S(\qx)-\qc}$ are given as
\begin{align}\label{eq:ls}
  \qx = \qx[c] + \lambda\,\qx[1] + \mu\,\qx[1](\imb),
\end{align}
yielding a minimal value of $\Abs{(I-P)(\qc)}$, that is, the distance from
$\qc$ to $\ran S$. In particular, for $\lambda = \mu = 0$ it holds $\qx = \qx[c]$, which is the
unique least-squares solution of minimal norm.
\item If \eqref{eq:INHOMOGENEOUS_SINGULAR_SYLVESTER_EQUATION} has a solution,
then \eqref{eq:ls} is its general solution and $\qx = \qx[c]$ with
$S(\qc) = 2(\ima)\qc$ is the unique solution of minimal norm. In that case it holds
\begin{align}\label{eq:minnorm}
  \Abs{\qx[c]} = \frac{\Abs{\qc}}{2\Abs{\ima}},
\end{align}
and $\qx[c]$ is orthogonal to $\ker S$ such that every solution
\eqref{eq:ls} satisfies
\begin{align}\label{eq:norm}
  \Abs{\qx}^2 = \frac{\Abs{\qc}^2}{4\Abs{\ima}^2}
              + \bigl(\lambda^2 + \mu^2\Abs{\ima}^2\bigr)\Abs{\qx[1]}^2 .
\end{align}
\end{enumerate}
\end{theorem}

\begin{proof}
By Corollary~\ref{corollary:SINGULAR_GENERAL_LINEAR_FUNCTION}, the
quaternions $\qa,\qb$ may, without loss of generality, be considered as
pure.
For statement (i) first note that by Theorem~\ref{theorem:MP} we have
\begin{align*}
  P(\qc) \in \ran S, \qquad (I-P)(\qc) \in \ker S,
\end{align*}
and $S(\qx[c]) = SS^{+}(\qc) = P(\qc)$ such that for an arbitrary quaternion $\qx$ it follows
\begin{align*}
  S(\qx) - P(\qc) = S(\qx) - S(\qx[c]) = S(\qx - \qx[c]) \in \ran S.
\end{align*}
Then, since $\Hset=\ker S \oplus \ran S$ by Corollary~\ref{corollary:H_IS_DIRECT_ORTHOGONAL_SUM},
it holds
\begin{align*}
  \Abs{S(\qx) - \qc}^2 
  &= \Abs[\big]{[S(\qx) - P(\qc)] - [(I-P)(\qc)]}^2\\
  &= \Abs{S(\qx - \qx[c])}^2 - 2\ip{S(\qx - \qx[c])}{(I-P)(\qc)} + \Abs{(I-P)(\qc)}^2\\
  &= \Abs{S(\qx - \qx[c])}^2 + \Abs{(I-P)(\qc)}^2.
\end{align*}
Hence, $\Abs{S(\qx)-\qc}$ is minimal if and only if $\qx - \qx[c]\in\ker S$.
By Theorem~\ref{theorem:NONZERO_SOLUTIONS_OF_THE_HOMOGENEOUS_SYLVESTER_EQUATION},
this yields \eqref{eq:ls}. The minimal value is $\Abs{(I-P)(\qc)}$, which
is the distance from $\qc$ to $\ran S$, since $S(\qx)$ ranges over $\ran S$.
Further, since
\begin{align*}
  \qx[c] = S^{+}(\qc) = -\frac{S(c)}{4\Abs{\qa}^2}
         = S\Bigl(-\frac{\qc}{4\Abs{\qa}^2}\Bigr) \in \ran S,
\end{align*}
$\qx[c]$ is orthogonal to $\ker S$, and by 
Theorem~\ref{theorem:NONZERO_SOLUTIONS_OF_THE_HOMOGENEOUS_SYLVESTER_EQUATION}
the quaternions $\qx[1], \qx[1]\qb \in \ker S$
are orthogonal. Hence, every $\qx$ of the form \eqref{eq:ls} satisfies
\begin{align}
  \Abs{\qx}^2
  = \Abs{\qx[c]}^2 + \bigl(\lambda^2 + \mu^2\Abs{\qa}^2\bigr)\Abs{\qx[1]}^2,
  \label{eq:NORM_OF_LS_SOLUTIONS}
\end{align}
which is minimal if and only if $\lambda = \mu = 0$.

For statement (ii) assume that
\eqref{eq:INHOMOGENEOUS_SINGULAR_SYLVESTER_EQUATION} has a solution $\qx$.
Then, the minimal value of $\Abs{S(\qx)-\qc}$ is zero and $\qc\in\ran S$ such that 
$\qc = P(\qc) = SS^{+}(\qc) = S(\qx[c])$ and 
\begin{align*}
  0 = S(\qx) - \qc = S(\qx) - S(\qx[c]) = S(\qx - \qx[c]),
\end{align*}
that is, $\qx - \qx[c]\in\ker S$. Hence, by 
Theorem~\ref{theorem:NONZERO_SOLUTIONS_OF_THE_HOMOGENEOUS_SYLVESTER_EQUATION} 
the quaternions \eqref{eq:ls} are exactly the solutions of 
\eqref{eq:INHOMOGENEOUS_SINGULAR_SYLVESTER_EQUATION}.
Further, by Theorem~\ref{theorem:SOLVABILITY} we have
$\qa\qc = -\qc\qb$ and hence $S(\qc) = 2\qa\qc$, such that
\begin{align}
  \qx[c] = -\frac{S(\qc)}{4\Abs{\qa}^2} = -\frac{\qa\qc}{2\Abs{\qa}^2},
  \qquad
  \Abs{\qx[c]} = \frac{\Abs{\qc}}{2\Abs{\qa}}.
  \label{eq:NORM_XC}
\end{align}
Reclaiming the final part of the proof of statement (i) and substituting
\eqref{eq:NORM_XC} into \eqref{eq:NORM_OF_LS_SOLUTIONS} yields
\eqref{eq:norm} such that $\qx = \qx[c]$ is the unique solution of minimal norm.
\end{proof}

\begin{corollary}\label{cor:tian}
The first term of Tian's general solution \eqref{eq:TIAN_INHOMOGENEOUS_SOLUTION}
is exactly the minimum-norm solution $\qx[c]$ of
\eqref{eq:INHOMOGENEOUS_SINGULAR_SYLVESTER_EQUATION}. The remaining term of
\eqref{eq:TIAN_INHOMOGENEOUS_SOLUTION} satisfies
\begin{align*}
  \qp - \frac{1}{\Abs{\ima}^2}\,(\ima)\qp(\imb) = 2\,(I-P)(\qp),
\end{align*}
with $\qx=\qp$ by Theorem~\ref{theorem:MP}, ranges over $\ker S$, and is orthogonal to $\qx[c] \in\ran S$.
\end{corollary}

\begin{example}\label{ex:1}
Let $\qa = \qi$, $\qb = \qj$, and $\qc = \qi - \qj$. Then $\qa \sim \qb$ and it holds
\begin{align*}
  \qa\qc = \qi(\qi - \qj) = -1 - \qk = (\qi - \qj)(-\qj) = \qc\qb[][*],
\end{align*}
such that \eqref{eq:INHOMOGENEOUS_SINGULAR_SYLVESTER_EQUATION} has a solution
by Theorem~\ref{theorem:SOLVABILITY}. By \eqref{eq:xc} and \eqref{eq:minnorm} its
minimum-norm solution and the norm of this solution are given as
\begin{align*}
  \qx[c] = \frac{(\qi - \qj)\qj - \qi(\qi - \qj)}{4} = \frac{1 + \qk}{2},
  \qquad
  \Abs{\qx[c]} = \frac{\Abs{\qc}}{2\Abs{\ima}} = \frac{1}{\sqrt{2}}.
\end{align*}
Since $\ima \ne -\imb$,
Theorem~\ref{theorem:NONZERO_SOLUTIONS_OF_THE_HOMOGENEOUS_SYLVESTER_EQUATION}
yields $\qx[1] = \qi + \qj$ and $\qx[1](\imb) = (\qi + \qj)\qj = -1 + \qk$,
which are both orthogonal to $\qc$, as required by
Theorem~\ref{theorem:SOLVABILITY}\,(iii). Hence, by Theorem~\ref{theorem:general}\,(ii) the
general solution is
\begin{align*}
  \qx = \frac{1 + \qk}{2} + \lambda(\qi + \qj) + \mu(-1 + \qk),
  \qquad
  \Abs{\qx}^2 = \frac{1}{2} + 2(\lambda^2 + \mu^2),
\end{align*}
for arbitrary real numbers $\lambda, \mu$.

Next, let $\qa = \qi$ and $\qb = -\qi = \qa[][*]$ such that $\ima = -\imb$.
Then Theorem~\ref{theorem:NONZERO_SOLUTIONS_OF_THE_HOMOGENEOUS_SYLVESTER_EQUATION}
with $\qx[1] = \qj$ yields $\qx[1](\imb) = \qj(-\qi) = \qk$ such that
$\ker S$ is spanned by $\qj$ and $\qk$. By Theorem~\ref{theorem:SOLVABILITY}\,(ii),
\eqref{eq:INHOMOGENEOUS_SINGULAR_SYLVESTER_EQUATION} has a solution if and
only if $\qi\qc = \qc\qi$, that is, if and only if
$\qc = \nu_0 + \nu_1\qi$ with real numbers $\nu_0, \nu_1$. For $\qc = 1$,
\eqref{eq:xc} yields the minimum-norm solution
$\qx[c] = (-\qi - \qi)/4 = -\qi/2$ with
$\Abs{\qx[c]} = 1/2 = \Abs{\qc}/(2\Abs{\ima})$, and the general solution is
$\qx = -\qi/2 + \lambda\qj + \mu\qk$ for arbitrary real numbers $\lambda, \mu$. In contrast, for $\qc = \qj$ it holds
$\qa\qc = \qk \ne -\qk = \qc\qb[][*]$ such that
\eqref{eq:INHOMOGENEOUS_SINGULAR_SYLVESTER_EQUATION} has no solution. In that
case, \eqref{eq:xc} yields $\qx[c] = \bigl(\qj(-\qi) - \qi\qj\bigr)/4 = 0$.
Hence, by Theorem~\ref{theorem:general}\,(i) the least-squares solutions are
$\qx = \lambda\qj + \mu\qk$. Since $\qc = \qj \in \ker S$ and $I-P$ is the
orthogonal projector onto $\ker S$ by Theorem~\ref{theorem:MP}, it holds
$(I-P)(\qc) = \qc$ such that the minimal value is $\Abs{\qc} = 1$.
\end{example}

\section{Conclusion}\label{sec:CONCLUSION}

We have treated the singular quaternionic Sylvester equation
$\qa\qx-\qx\qb=\qc$ by working directly with the real-linear function
$S(\qx)=\qa\qx-\qx\qb$ complementing the real matrix representation.
The solution space of the homogeneous equation is the complex line
$\qx[1]\C[\qb]$ through a nonzero solution $\qx[1]$, with an explicit
orthogonal basis that also covers the case $\qb=\qa[][*]$ excluded from
Tian's formula. Building on this, the skew-symmetry of $S$ and the identity
$S^3=-4\Abs{\ima}^2S$ carry the theory of the inhomogeneous equation:
the equation is solvable if and only if the right-hand side is orthogonal to the solution space of the
homogeneous equation, which is equivalent to Tian's condition
$\qa\qc=\qc\qb[][*]$; the Moore--Penrose inverse of $S$ is
$-S/(4\Abs{\ima}^2)$; and the first term of Tian's general solution is the minimum-norm solution if the equation is solvable and the minimum-norm least-squares solution otherwise. Adding the solutions of the homogeneous equation yields all solutions in
the first case and all least-squares solutions in the second. This answers the questions raised in
\cite[Remark~4.2]{tian2004_equations} for the equation $\qa\qx-\qx\qb=\qc$.
The explicit basis of the solution space is used in
\cite{radak+scheunert++2026_closed-form} to solve Wahba's attitude
determination problem for two vector observations within the quaternion
algebra. The questions regarding $\qa\qx-\qx[][*]\qb=\qc$ and
$\qx[][*]\qa\qx=\qb$ are left for future work.

\section*{Funding}
This work was supported by the Federal Ministry of Research, Technology, and Space (BMFTR), Germany, as part of the research programme Communication Systems ``Souverän. Digital. Vernetzt.'', joint project 6G-life [project identification number 16KIS2413K]; and by the German Research Foundation (DFG, Deutsche Forschungsgemeinschaft) as part of Germany's Excellence Strategy -- EXC 2050/2 -- Project ID 390696704 -- Cluster of Excellence ``Centre for Tactile Internet with Human-in-the-Loop'' (CeTI) of TUD Dresden University of Technology. The funding sources had no involvement in the study design, the analysis, the writing of the article, or the decision to submit it for publication.

\section*{Declaration of competing interest}
The authors declare that they have no known competing financial interests or personal relationships that could have appeared to influence the work reported in this paper.

\section*{Data availability}
No data was used for the research described in the article.

\bibliographystyle{unsrt}
\bibliography{paper.bib}

\end{document}